\documentclass[12pt,twoside,a4paper]{amsart}

\usepackage{graphicx,amsthm,amsmath,geometry,hyperref}

\usepackage[usenames]{color}
\usepackage{amssymb}
\usepackage{amscd}
\usepackage{amsfonts}
\usepackage{enumitem}
\usepackage{stmaryrd}

\theoremstyle{plain}
\newtheorem{theorem}{Theorem}
\newtheorem{corollary}[theorem]{Corollary}

\theoremstyle{definition}
\newtheorem{definition}[theorem]{Definition}
\newtheorem{example}[theorem]{Example}

\theoremstyle{remark}
\newtheorem{remark}[theorem]{Remark}

\newcommand{\be}{\begin{equation}}
\newcommand{\ee}{\end{equation}}
\newcommand{\beba}{\begin{equation}\begin{array}{lcl}}
\newcommand{\eaee}{\end{array}\end{equation}}
\newcommand{\bea}{\begin{eqnarray}}
\newcommand{\eea}{\end{eqnarray}}
\newcommand{\ba}{\begin{array}}
\newcommand{\ea}{\end{array}}
\newcommand{\ktt}{T_0T^\ast}
\numberwithin{equation}{section}

\newcommand{\seqnum}[1]{\href{http://oeis.org/#1}{\underline{#1}}}

\begin{document}

\title[Generalized Perfect Numbers]{Some Results on Generalized Multiplicative Perfect Numbers}

\author[A. Laugier]{Alexandre Laugier}
\address{Lyc{\'e}e professionnel Tristan Corbi{\`e}re, 16 rue de
Kerveguen - BP 17149, 29671 Morlaix cedex, France} 
\email{laugier.alexandre@orange.fr}

\author[M. P. Saikia]{Manjil P. Saikia}
\footnotetext[1]{Corresponding Author: manjil@gonitsora.com.}
\address{Universit\"at Wien, Fakult\"at f\"ur Mathematik, Oskar-Morgenstern-Platz 1, 1090 Wien, Austria} 
\email{manjil@gonitsora.com, manjil.saikia@univie.ac.at}

\author[U. Sarmah]{Upam Sarmah}
\address{HS Student, Nowgong College, Nagaon-782001, Assam, India.}
\email{upam.sarmah@gmail.com}



\begin{abstract}
In this article, based on ideas and results by J. S\'andor \cite{sandor1, sandor}, we define $k$-multiplicatively $e$-perfect numbers and $k$-multiplicatively 
$e$-superperfect numbers and prove some results on them. We also characterize the $k$-$\ktt$-perfect numbers defined by Das and Saikia \cite{nntdm} in details.
\end{abstract}

\maketitle

\section{{Introduction}}

A natural number $n$ is said to be \textbf{perfect} (\seqnum{A000396}) if the sum of all proper divisors of $n$ is equal to $n$. Or equivalently, $\sigma(n)=2n$, where $\sigma(k)$ is the sum of the divisors of $k$. It is a well known result of Euler-Euclid that the form of even perfect numbers is $n=2^kp$, where $p=2^{k+1}-1$ is a Mersenne prime and $k\geq 1$. Till date, no odd perfect number is known, and it is believed that none exists. Moreover $n$ is said to be \textbf{super-perfect} if $\sigma(\sigma(n))=2n$. It was proved by Suryanarayana-Kanold \cite{kanold, surya} that the general form of such super-perfect numbers are $n=2^k$, where $2^{k+1}-1$ is a Mersenne prime and $k\geq 1$. No odd super-perfect numbers are known till date. Unless, otherwise mentioned all $n$ considered in this paper will be a natural number and $d(n)$ will be the number of divisors of $n$. We also denote by $\llbracket 1, n \rrbracket$, the set $\{1, 2, \ldots, n\}$ and by $\mathbb{N}^\ast$, the set $\mathbb{N}\cup \{0\}$.

Let $T(n)$ denote the product of all the divisors of $n$. A \textbf{multiplicatively perfect number} (\seqnum{A007422}) is a number $n$ such that $T(n)=n^2$ and $n$ is called \textbf{multiplicatively super-perfect} if $T(T(n))=n^2$. S\'andor \cite{sandor1} characterized such numbers and also numbers called \textbf{k-multiplicatively perfect numbers}, which are numbers $n$, such that $T(n)=n^k$ for $k\geq 2$. In this article we shall give some results on other classes of perfect numbers defined by various authors as well as by us. For a general introduction to such numbers, we refer the readers to subsection 1.11 in S\'andor and Crstici's book \cite[p. 55 -- 58]{sandornt} and to the article \cite{nntdm}, which contains various references to the existing literature.

\section{{\texorpdfstring{$k$}{k}-multiplicatively \texorpdfstring{$e$}{e}-perfect and superperfect numbers}}

S\'andor \cite{sandor} studied the multiplicatively $e$-perfect numbers defined below. If $n=p_1^{a_1}\cdots p_r^{a_r}$ is the prime factorization of $n>1$, a divisor $d$ of $n$ is called an \textbf{exponential divisor} (or, $e$-divisor for short) if $d=p_1^{b_1}\cdots p_r^{b_r}$ with $b_i\mid a_i$ for $i=1, \ldots, r$. This notion is due to Straus and Subbarao \cite{subba}. Let $\sigma_e(n)$ denote the sum of $e$-divisors of $n$, then Straus and Subbarao define $n$ as \textbf{exponentially perfect} (or, $e$-perfect for short) (\seqnum{A054979}) if $\sigma_e(n)=2n$. They proved that there are no odd $e$-perfect numbers, and for each $r$, such numbers with $r$ prime factors are finite. We refer the reader to the article \cite{sandor} for some historical comments and results related to such $e$-perfect numbers. In \cite{sandor3}, S\'andor 
also studied some type of $e$-harmonic numbers. An integer $n$ is called \textbf{$e$-harmonic of type $1$} if $\sigma_e(n)|nd_e(n)$, where $\sigma_e(n)$ (resp. $d_e(n)$) is the sum 
(resp. number) of $e$-divisors of $n$. It is easy to check that $$d_e(n)=d(a_1)\cdots d(a_r).$$ S\'andor \cite{sandor3} also defined $n$ to be \textbf{$e$-harmonic of type $2$} if 
$S_e(n)|nd_e(n)$, where \[S_e(n)=\prod_{i=1}^r\left(\sum_{d_i|a_i}p_i^{a_i-d_i}\right).\]

Let $T_e(n)$ denote the product of all the $e$-divisors of $n$. Then $n$ is called \textbf{multiplicatively $e$-perfect} if $T_e(n)=n^2$ and \textbf{multiplicatively $e$-superperfect} if $T_e(T_e(n))=n^2$. The main result of S\'andor \cite{sandor} is the following.

\begin{theorem}[S\'andor, \cite{sandor} Theorem 2.1]\label{thms1}

$n$ is multiplicatively $e$-perfect if and only if $n=p^a$, where $p$ is a prime and $a$ is an ordinary perfect number. $n$ is multiplicatively $e$-superperfect if and only if $n=p^a$, where $p$ is a prime and $a$ is an ordinary superperfect number, that is $\sigma(\sigma(a))=2a$.

\end{theorem}

We now give two result on $e$-harmonic numbers below, before we proceed to the main goal of our paper, that is to define and characterize some other classes of numbers.

\begin{theorem}\label{thmo1}
 If $n$ is multiplicatively $e$-perfect or multiplicatively $e$-superperfect, then $n$ is $e$-harmonic of type $1$ if and only if $\sigma_e(n)/p|d_e(n)$, where $p$ is the prime 
 as described in Theorem \ref{thms1}.
\end{theorem}

\begin{proof}
 We prove the result for the case when $n$ is multiplicatively $e$-perfect, the other case is similar. By Theorem \ref{thms1}, $n=p^a$ where $p$ is a prime and $a$ is an ordinary perfect number. So, 
 for $n$ to be $e$-harmonic of type $1$, we must have $\sigma_e(n)|nd_e$, which is enough to verify our claim.

\end{proof}

\begin{theorem}\label{thmo2}
 If $n$ is multiplicatively $e$-perfect or multiplicatively $e$-superperfect, then $n$ is $e$-harmonic of type $2$ if and only if $S_e(n)|d_e(n)$.
\end{theorem}

\noindent We skip the proof of Theorem \ref{thmo2}, as it is similar to the proof of Theorem \ref{thmo1} and uses Theorem \ref{thms1} in a similar way.

Inspired by the work of others in introducing generalized multiplicative perfect numbers, we now introduce the following two classes of numbers which are a natural generalization to the concept of $e$-perfect numbers..

\begin{definition}
A natural number $n$ is called \textbf{$k$-multiplicatively $e$-perfect} if $T_e(n)=n^k$, where $k\geq 2$. 
\end{definition}

\begin{definition}

A natural number $n$ is called \textbf{$k$-multiplicatively $e$-superperfect number} if $T_e(T_e(n))=n^k$, where $k\geq 2$.
\end{definition}

Before proceeding, we note from S\'andor \cite{sandor} that 
\begin{equation}\label{lt}
T_e(n)=p_1^{\sigma(a_1)d(a_2)\cdots d(a_r)}\cdots p_r^{\sigma(a_r)d(a_1)\cdots d(a_{r-1})}.
\end{equation}

\noindent S\'andor \cite{sandor2} also gave an alternate expression for $T_e(n)$ in terms of the arithmetical function $t(n)$ defined as $$t(n)=p_1^{2\frac{\sigma(a_1)}{d(a_1)}}\cdots p_r^{2\frac{\sigma(a_r)}{d(a_r)}}$$ with $t(1)=1$. We have from S\'andor \cite{sandor2}
\begin{equation}\label{ntuse}
T_e(n)=(t(n))^{d_e(n)/2}.
\end{equation}
\noindent We however, do not use \eqref{ntuse} in this note as we are only interested in the canonical forms of the numbers we have so far defined.

Before we characterize these classes of numbers we work on a few examples. 

\begin{example}
There exist $6k$-multiplicatively $e$-perfect numbers
with $k\in\mathbb{N}^{\star}$, which have the form
$(p_1\cdot p_2\cdot p_3)^{\alpha}$ provided 
$$
\sigma(\alpha)d(\alpha)^2=6k\alpha
$$

\noindent Thus, by \eqref{lt}, we have

$$
T_e((p_1\cdot p_2\cdot p_3)^{\alpha})=(p_1\cdot p_2\cdot p_3)^{\sigma(\alpha)d(\alpha)^2}.
$$

\end{example}

\begin{example}
There exist $k$-multiplicatively $e$-superperfect numbers for a nonzero positive even number $k$. For instance, by a routine application of \eqref{lt}  it can be verified that there exist
\begin{enumerate}[label=$\bullet$]
\item $20$-multiplicatively $e$-superperfect numbers which have the form $(p_1\cdot p_2)^3$;
\item $24$-multiplicatively $e$-superperfect numbers which have the form $(p_1\cdot p_2)^2$;
\item $32$-multiplicatively $e$-superperfect numbers which have the form $(p_1\cdot p_2)^4$;
\item $48$-multiplicatively $e$-superperfect numbers which have the form $(p_1\cdot p_2)^{16}$;
\item $64$-multiplicatively $e$-superperfect numbers which have the form $(p_1\cdot p_2)^{64}$;
\item $110$-multiplicatively $e$-superperfect numbers which have the form $(p_1\cdot p_2)^{93}$;
\item $168$-multiplicatively $e$-superperfect numbers which have the form $p^{6}_1\cdot p^{16}_2$ or $(p_1\cdot p_2)^{27}$;
\item $216$-multiplicatively $e$-superperfect numbers which have the form $(p_1\cdot p_2)^{14}$;
\item $234$-multiplicatively $e$-superperfect numbers which have the form $(p_1\cdot p_2)^{10}$;
\item $252$-multiplicatively $e$-superperfect numbers which have the form $(p_1\cdot p_2)^8$.
\end{enumerate}

In the following, we explain some of these examples. By analysis of the cases above, we notice that there exist for nonzero positive integers $m$ such that $2^{m+1}-1$ is prime (and so $m+1$ is prime), $8(m+2)$-multiplicatively $e$-superperfect numbers which have the form $(p_1\cdot p_2)^{2^m}$. Indeed, according to \eqref{lt}, we have
$$
T_e((p_1\cdot p_2)^{2^m})=(p_1\cdot p_2)^{\sigma(2^m)\cdot d(2^m)}
=(p_1\cdot p_2)^{(2^{m+1}-1)\cdot(m+1)}
$$
and since we assume that $m+1$ and $2^{m+1}-1$ are prime, we have

\begin{align*}
T_e(T_e((p_1\cdot p_2)^{2^m})) &= (p_1\cdot p_2)^{\sigma((2^{m+1}-1)\cdot(m+1))\cdot d((2^{m+1}-1)\cdot(m+1))}\\
 &= (p_1\cdot p_2)^{2^{m+1}\cdot(m+2)\cdot 4}=((p_1\cdot p_2)^{2^m})^{8(m+2)}.
\end{align*}

Moreover, by analysis of the cases above, we can notice that there exist for some odd prime number $p$, $2p$-multiplicatively $e$-superperfect numbers which have the form $(p_1\cdot p_2)^{2p}$. To see this, by \eqref{lt}, we have
$$
T_e((p_1\cdot p_2)^{2p})=(p_1\cdot p_2)^{\sigma(2p)\cdot d(2p)}
=(p_1\cdot p_2)^{3\cdot(p+1)\cdot 4}=(p_1\cdot p_2)^{12\cdot(p+1)}.
$$
If we can represent $p$ as $p=2^m q-1$ with $m$ a nonzero positive integer and $q$ an odd positive integer, then we have
$$
T_e((p_1\cdot p_2)^{2p})=(p_1\cdot p_2)^{2^{m+2}q\cdot 3}.
$$

If $q=3$, we have
$$
T_e(T_e((p_1\cdot p_2)^{2p}))=(p_1\cdot p_2)^{\sigma(2^{m+2}\cdot 3^2)\cdot d(2^{m+2}\cdot 3^2)}=(p_1\cdot p_2)^{(2^{m+3}-1)\cdot 13\cdot(m+3)\cdot 3}.
$$
We search a solution such that $2^{m+3}-1$ is divisible by $p=2^m\cdot 3-1$. That is, an integer $x\geq 1$ such that
$$
2^{m+3}-1=xp=x(2^m\cdot 3-1).
$$
This gives

\begin{align*}
    2^m\cdot 8-1 &= xp\\
    \Rightarrow 3\cdot 2^m\cdot 8-3 &= 3xp\\
    \Rightarrow (p+1)\cdot 8-3 &= 3xp\\
    \Rightarrow (3x-8)\cdot p &= 5.
\end{align*}
Since $p$ is a prime this implies that $p=5$ and $x=3$. Therefore, $m=1$ and we recover that
$$
T_e(T_e((p_1\cdot p_2)^{10}))=(p_1\cdot p_2)^{15\cdot 13\cdot 12}
=((p_1\cdot p_2)^{10})^{234}.
$$

If $q$ is not divisible by $3$, then we have
\begin{align*}
T_e(T_e((p_1\cdot p_2)^{2p})) &= (p_1\cdot p_2)^{\sigma(2^{m+2}q\cdot 3)\cdot d(2^{m+2}q\cdot 3)}=(p_1\cdot p_2)^{(2^{m+3}-1)\cdot\sigma(q) 4\cdot(m+3)\cdot d(q)\cdot 2}\\
 &= (p_1\cdot p_2)^{8(m+3)\cdot \sigma(q)\cdot d(q)\cdot(2^{m+3}-1)}.
\end{align*}

We search a solution such that $2^{m+3}-1$ is divisible by $p$. That is, an integer $y\geq 1$ such that
$$
2^{m+3}-1=yp=y(2^m q-1).
$$
This will reduce to $$
(qy-8)\cdot p=8-q.
$$
If $q>7$, then $8-q<0$ where as $(qy-8)\cdot p>0$. We reach a contradiction meaning that the only possible values of $q$ are $1,5,7$.
If $q=1$, then we have $(y-8)\cdot p=7$. This implies that $p=7, y=9$ and $m=3$ and so on
$$
T_e(T_e((p_1\cdot p_2)^{2p}))=(p_1\cdot p_2)^{8\cdot 6\cdot 63}
=((p_1\cdot p_2)^{14})^{216}.
$$
If $q=5, 7$, then we get no integral solutions for $y$.
\end{example}

Now, we characterize some of these classes of numbers in the following theorems. Note that Theorem \ref{t2-2} is analogous to Theorem \ref{thms1} of S\'andor.

\begin{theorem}\label{t2-2}
If $n=p^a$, where $p$ is a prime and $a$ is a $k$-perfect number, then
$n$ is $k$-multiplicatively $e$-perfect. If $n=p^a$, where $p$ is a
prime and $a$ is a $k$-superperfect number, then 
$n$ is $k$-multiplicatively $e$-superperfect.
\end{theorem}

\begin{proof}
We know from S\'andor \cite{sandor}, that if $n=p^a$ where $p$ is a prime and $a$ is a
non-zero positive integer, then we have
$$
T_e(n)=p^{\sigma(a)}.
$$
So, if $n=p^a$ where $p$ is a prime and $a$ is a $k$-perfect number,
then
$$
\sigma(a)=ka,
$$
and so
$$
T_e(n)=p^{ka}=n^k.
$$
Moreover, if $n=p^a$ where $p$ is a prime and $a$ is a non-zero
positive integer, then we have
$$
T_e(T_e(n))=T_e(p^{\sigma(a)})=p^{\sigma(\sigma(a))}.
$$
So, if $n=p^a$ where $p$ is a prime and $a$ is a $k$-perfect number,
then
$$
\sigma(\sigma(a))=ka,
$$
and so
$$
T_e(T_e(n))=p^{ka}=n^k.
$$
\end{proof}

\begin{theorem}\label{t2-3}
Let $p$ be a prime number and let $n=p^{\alpha_1}_1\cdots p^{\alpha_r}_r$, with
$r\in\mathbb{N^\ast}$ be the prime factorisation
of an integer $n>1$ where $p_i$ with $i\in\llbracket 1,r\rrbracket$
are prime numbers and $\alpha_i\in\mathbb{N^\ast}$ for all
$i\in\llbracket 1,r\rrbracket$. $n$ is $p$-multiplicatively
$e$-perfect if and only if for each $i\in\llbracket 1,r\rrbracket$, we have
$$
\sigma(\alpha_i)=\gcd(\alpha_i,\sigma(\alpha_i))p
$$
and
$$
\alpha_i=\gcd(\alpha_i,\sigma(\alpha_i))
{\displaystyle\prod_{j\in\llbracket 1,r\rrbracket\setminus\{i\}}}d(\alpha_j),
$$
with
$$
2^{r-1}\leq{\displaystyle\prod_{j\in\llbracket 1,r\rrbracket\setminus\{i\}}}d(\alpha_j)<p,
$$
where we set
$$
{\displaystyle\prod_{j\in\emptyset}}d(\alpha_j)=1.
$$
In particular, if $r=1$, then $\alpha_1|\sigma(\alpha_1)$ and we have
$$
\sigma(\alpha_1)=\alpha_1p.
$$
\end{theorem}

\begin{proof}

If $r=1$, then using Theorem \ref{t2-2}, $n=p^{\alpha_1}_1$ is
$p$-multiplicatively $e$-perfect if and only if $\alpha_1$ is a
$p$-perfect number.

We now assume that $r\geq 2$. We have 
\be\label{E2-1}
T_e(n)=n^p\,\Leftrightarrow\,\left\{\begin{array}{c}
\sigma(\alpha_1) d(\alpha_2)\cdots d(\alpha_r)=p\alpha_1;
\\
\vdots
\\
\sigma(\alpha_r) d(\alpha_1)\cdots d(\alpha_{r-1})=p\alpha_r.
\end{array}\right.
\ee
Clearly, if for each $i\in\llbracket 1,r\rrbracket$, we have
$$
\sigma(\alpha_i)=\gcd(\alpha_i,\sigma(\alpha_i))p
$$
and
$$
\alpha_i=\gcd(\alpha_i,\sigma(\alpha_i))
{\displaystyle\prod_{j\in\llbracket 1,r\rrbracket\setminus\{i\}}}d(\alpha_j)
$$
then it can be easily verified that (\ref{E2-1}) is satisfied.

Now we notice that if all $\alpha_i$ ($i\in\llbracket 1,r\rrbracket$) are
equal to $1$, then (\ref{E2-1}) is consistent only if
$p=1$. It is impossible since $p$ is a prime number.
If at least an $\alpha_i$ is equal to $1$, say $\alpha_1=1$, then from
(\ref{E2-1}), we have
$$
d(\alpha_2)\cdots d(\alpha_r)=p.
$$
Then one of $d(\alpha_2),\ldots,d(\alpha_r)$ is equal to the prime
$p$ and the others are equal to $1$. Say
$$
d(\alpha_2)=p,
$$
and
$$
d(\alpha_3)=\cdots=d(\alpha_r)=1.
$$
It implies that
$$
\alpha_1=\alpha_3=\cdots=\alpha_r=1.
$$
Then the equation $\sigma(\alpha_2)d(\alpha_1)d(\alpha_3)\cdots
d(\alpha_r)=p\alpha_2$ of (\ref{E2-1}) gives $\sigma(\alpha_2)=p\alpha_2$. But, for all $i\in\llbracket 1,r\rrbracket\setminus\{1,2\}$,
the equation $\sigma(\alpha_i)d(\alpha_1)d(\alpha_2)\cdots 
d(\alpha_r)=p\alpha_i$ of (\ref{E2-1}) gives
$\sigma(\alpha_i)=\alpha_i$ which is not possible since we know that $\sigma(\alpha_i)\geq\alpha_i+1$. So, we must have $\alpha_i\geq 2$ for all $i\in\llbracket 1,r\rrbracket$. Thus,
$$
{\displaystyle\prod_{j\in\llbracket 1,r\rrbracket\setminus\{i\}}}d(\alpha_j)
\geq 2^{r-1}.
$$

We now prove that if (\ref{E2-1}) is true, then
for each $i\in\llbracket 1,r\rrbracket$, we have
$$
\sigma(\alpha_i)=\gcd(\alpha_i,\sigma(\alpha_i))p
$$
and
$$
\alpha_i=\gcd(\alpha_i,\sigma(\alpha_i))
{\displaystyle\prod_{j\in\llbracket 1,r\rrbracket\setminus\{i\}}}d(\alpha_j).
$$
Let $g_i=\gcd(\alpha_i,\sigma(\alpha_i))$ for all
$i\in\llbracket 1,r\rrbracket$. So, for each $i\in\llbracket
1,r\rrbracket$, there exists two non-zero positive integer $a_i,s_i$
such that $\alpha_i=g_ia_i$ and $\sigma(\alpha_i)=g_is_i$ with $\gcd(a_i,s_i)=1$. Notice that $s_i>1$. Otherwise, we would have $\sigma(\alpha_i)|\alpha_i$, a contradiction.

For each $i\in\llbracket 1,r\rrbracket$, the equation
$$
\sigma(\alpha_i)
{\displaystyle\prod_{j\in\llbracket 1,r\rrbracket\setminus\{i\}}}d(\alpha_j)=p\alpha_i
$$
of (\ref{E2-1}) now gives
$$
s_i{\displaystyle\prod_{j\in\llbracket 1,r\rrbracket\setminus\{i\}}}d(\alpha_j)=pa_i.
$$
Since $\gcd(a_i,s_i)=1$, then from Euclid's lemma we get $a_i|{\displaystyle\prod_{j\in\llbracket 1,r\rrbracket\setminus\{i\}}}d(\alpha_j)$, and $s_i|p$. So, there exists an integer $k$ such that ${\displaystyle\prod_{j\in\llbracket 1,r\rrbracket\setminus\{i\}}}d(\alpha_j)=ka_i$, and $p=ks_i$.

Using the fact that $p$ is a prime and $s_i>1$, we have that $k=1$
and we get ${\displaystyle\prod_{j\in\llbracket 1,r\rrbracket\setminus\{i\}}}d(\alpha_j)=a_i$, and $p=s_i$. Therefore $\sigma(\alpha_i)=g_ip$ and $\gcd(p,a_i)=1$. Moreover, since $\sigma(\alpha_i)\geq\alpha_i+1$, we have
$$
g_i(p-a_i)\geq 1
$$
implying that $p>a_i$.
\end{proof}

\begin{example}
Let $n=p^{\alpha}_1\cdot p^{\alpha}_2$ be the prime factorisation of an
integer $n>1$ where $p_1$ and $p_2$ are prime numbers. Let $\alpha=18$. So, we have 
$$
d(\alpha)=d(2\cdot 3^2)=6\qquad\sigma(\alpha)=1+2+3+6+9+18=39=3\cdot 13=3p
$$
where $p=13$. Moreover, we have also
$$
\gcd(\alpha,\sigma(\alpha))=\gcd(18,39)=\gcd(3\cdot 6,3\cdot 13)
=3\cdot\gcd(6,13)=3\neq\alpha
$$
and
$$
\gcd(\alpha,\sigma(\alpha))\cdot d(\alpha)=3\cdot 6=18=\alpha.
$$
At this stage, we notice that Theorem \ref{t2-3} can be applied. Let us verify that
it is the case. We have
$$
T_e(n)=p^{\sigma(18)\cdot d(18)}_1\cdot p^{\sigma(18)\cdot d(18)}_2=p^{39\cdot 6}_1\cdot
p^{39\cdot 6}_2=p^{3\cdot 13\cdot 6}_1\cdot p^{3\cdot 13\cdot 6}_2
=(p^{18}_1\cdot p^{18}_2)^{13}=n^p.
$$
So, for all primes $p_1$, $p_2$, integers of the form
$p^{18}_1\cdot p^{18}_2$ are $13$-multiplicatively $e$-perfect
numbers.
\end{example}

\begin{example}
Let $n=p^{\alpha}_1\cdot p^{\alpha}_2\cdot p^{\alpha}_3$ be the
prime factorisation of an integer $n>1$ where $p_1$, $p_3$ and $p_2$
are prime numbers. Let $\alpha=9$. So, we have 
$$
d(\alpha)=d(3^2)=3\qquad\sigma(\alpha)=1+3+9=13=1\cdot 13=1\cdot p.
$$
where $p=13$. Moreover, we have also
$$
\gcd(\alpha,\sigma(\alpha))=\gcd(9,13)=1\neq\alpha
$$
and
$$
\gcd(\alpha,\sigma(\alpha))\cdot d(\alpha)^2=1\cdot 3^2=9=\alpha
$$
At this stage, we notice that Theorem \ref{t2-3} can be applied. Let us verify that
it is well the case. We have
$$
T_e(n)=p^{\sigma(9)\cdot d(9)^2}_1\cdot p^{\sigma(9)\cdot d(9)^2}_2
\cdot p^{\sigma(9)\cdot d(9)^2}_3=p^{13\cdot 9}_1\cdot p^{13\cdot 9}_2
\cdot p^{13\cdot 9}_3=(p^9_1\cdot p^9_2\cdot p^9_3)^{13}=n^p
$$
So, for all primes $p_1$, $p_2$ and $p_3$, integers of the form
$p^9_1\cdot p^9_2\cdot p^9_3$ are $13$-multiplicatively $e$-perfect
numbers.
\end{example}

\begin{remark}
Let $n=p^{\alpha_1}_1\cdots p^{\alpha_r}_r$ with
$r\in\mathbb{N^\ast}$ be the prime factorisation
of an integer $n>1$ where $p_i$ with $i\in\llbracket 1,r\rrbracket$
are prime numbers and $\alpha_i\in\mathbb{N^\ast}$ for all
$i\in\llbracket 1,r\rrbracket$. If all $\alpha_i$ are primes, then we
have
$$
\sigma(\alpha_i)=\alpha_i+1,
$$
and
$$
d(\alpha_i)=2.
$$
Notice that in such a case, Theorem \ref{t2-3} cannot be applied since
$\sigma(\alpha_i)$ is not divisible by $\alpha_i$ for all
$i\in\llbracket 1,r\rrbracket$. From \eqref{lt} we have,
$$
T_e(n)=p^{(\alpha_1+1)\cdot 2^{r-1}}_1\cdots p^{(\alpha_r+1)\cdot 2^{r-1}}_r,
$$
and so
$$
T_e(n)=n{\displaystyle\prod^r_{i=1}}p^{2^{r-1}}_i.
$$
If $r=1$, then $T_e(n)=np_1=p^{1+\alpha_1}$. If $r\geq 2$ and
if $\alpha_i=2$ for all $i\in\llbracket 1,r\rrbracket$, then $n$ is a
perfect square and we have
$$
n=(p_1\cdots p_r)^2,
$$
and
$$
T_e(n)=n^{1+2^{r-2}}.
$$
In particular, if $r=2$, then $T_e(n)=n^2$ meaning that $n$ is
multiplicatively $e$-perfect number. 
\end{remark}

We now prove a result related to the bounds on the prime $p$ in Theorem \ref{t2-3}. For that we will need the following results.

\begin{theorem}[Nicolas and Robin \cite{nc}]\label{ul7}
For $n\geq 3$,
\begin{equation*}
 \frac{\log d(n)}{\log 2}\leq C_1\frac{\log n}{\log\log n}
\end{equation*} where $C_1=1.5379\cdots$ with equality for $n= 2^5\cdot 3^3\cdot 5^2\cdot 7\cdot 11\cdot 13\cdot 17\cdot 19$.
\end{theorem}

\begin{theorem}[\cite{sandor22}, p. 77]\label{ul1}
For any natural number $n\geq 3$, $\sigma(n)<n\sqrt{n}$.
\end{theorem}

\begin{theorem}\label{ut3}
Let $n=p_1^{a_1}p_2^{a_2}\cdots p_r^{a_r}$ be the prime factorization of integer $n$, where $p_i, a_i, i\in\mathbb{N}, a_i\geq3$ and let $n$ be a $k$-multiplicatively-$e$-perfect number. Then,
we have
\begin{equation*}
2^{r-1}<k <  \prod_{i=1}^r \left( a_i^{0.5+\frac{C(r-1)}{\log \log a_i}}\right)^{\frac{1}{r}},
\end{equation*}
\noindent where $C=C_1\log 2$ and $C_1=1.5379\cdots$.
\end{theorem}

\begin{proof}
As $n$ is a $k$-multiplicatively-$e$-perfect number, so $n^k=p_1^{a_1k}p_2^{a_2k}\cdots p_r^{a_rk}=T_e(n)$. By using \eqref{lt} we get
\begin{align*}
\sigma(a_1)d(a_2)\cdots d(a_r) &= a_1k,\\
\vdots & \\
d(a_1)d(a_2)\cdots \sigma(a_r) &= a_rk.\\
\end{align*}
Multiplying these $r$ equalities we get,
\begin{equation}\label{ue4}
\sigma(a_1)\sigma(a_2)\cdots\sigma(a_r)d(a_1)^{r-1}d(a_2)^{r-1}\cdots d(a_r)^{r-1}=a_1a_2\cdots a_rk^r.
\end{equation}

Now we proceed to prove that $k>2^{r-1}$. For that notice,  for any $a_i$, $\sigma(a_i)\geq (a_i+1)$ and $d(a_i)\geq 2$. Thus we get the following inequality,
\begin{equation*}
\sigma(a_1)\cdots\sigma(a_r)d(a_1)^{r-1}\cdots d(a_r)^{r-1}\geq (a_1+1)(a_2+1)\cdots (a_r+1)2^{r-1}\cdots 2^{r-1}
\end{equation*}
Substituting the right hand side of \eqref{ue4} in the above inequality we get, $$a_1a_2\cdots a_rk^r\geq (a_1+1)(a_2+1)\cdots (a_r+1)(2^{r-1})^r$$ which implies $$(k/2^{r-1})^r\geq \frac{(a_1+1)(a_2+1)\cdots (a_r+1)}{a_1a_2\cdots a_r}>1.$$ This 
gives us $k>2^{r-1}$.

Now we proceed to set the upper bound. By Theorem \ref{ul7} we have $d(a_i)\leq a_i^{\frac{C}{\log \log a_i}}$, where $C= C_1\log 2$. Hence
\begin{equation}\label{uenew}
 \prod_{i=1}^r d(a_i)^{r-1}\leq \prod_{i=1}^r \left(a_i^{\frac{C}{\log \log a_i}}\right)^{r-1}.
\end{equation}

\noindent Again, by an application of Theorem \ref{ul1} we get

\begin{equation}\label{uenew2}
 \prod_{i=1}^r \sigma (a_i) < \prod_{i=1}^r a_i \sqrt{a_i}.
\end{equation}

Now, using \eqref{ue4}, \eqref{uenew} and \eqref{uenew2} we get \[k< \prod_{i=1}^r \left( a_i^{0.5+\frac{C(r-1)}{\log \log a_i}}\right)^{\frac{1}{r}}.\]

\end{proof}

\begin{example}
Let $m=p^6$, where $p$ is a prime. Now consider another arbitrary prime $q$. Then $n$ can be $q$-multiplicatively-e-perfect only for the primes in the interval 
$2^{1-1}<q\leq \left(6^{0.5+\frac{C_1\log 2\cdot (1-1)}{\log \log 6}}\right)^{\frac{1}{1}}$. That is $1<q<\sqrt 6$.
\end{example} 

The bounds on $p$ mentioned in Theorem \ref{ut3} are not tight and there is further scope to work on such bounds of primes. Moreover notice that the bounds mentioned here requires complete prime factorization of integer $m$. Hence, bounds that doesn't requires prime factorization of $m$ would be more efficient.  But we do not discuss this direction in the present paper.

\section{{\texorpdfstring{$k$}{k}-\texorpdfstring{$\ktt$}{k-TT}-perfect numbers}}

A divisor $d$ of $n$ is said to be \textbf{unitary} if $\gcd(d, n/d)=1$. Let $T^\ast(n)$ be the product of unitary divisors of $n$. Bege \cite{bege} has studied the multiplicatively unitary perfect numbers and proved results very similar to S\'andor. Das and Saikia \cite{nntdm} introduced the concept of \textbf{$T^\ast T$-perfect numbers} which are numbers $n$ such that $T^\ast(n)T(n)=n^2$. They also introduced $k$-$T^\ast T$-perfect numbers and characterized both these classes of numbers. 
They further introduced the concept of $T^\ast_0T$-superperfect and $k$-$T^\ast_0T$-perfect numbers. A number $n$ is called a \textbf{$T^\ast_0T$-superperfect number} if $T^\ast(T(n))=n^2$ and it is called a \textbf{$k$-$T^\ast_0T$-perfect number} if $T^\ast(T(n))=n^k$ for $k\geq 2$. Das and Saikia \cite{nntdm} characterized these classes of numbers. They also introduced the \textbf{$k$-$\ktt$-perfect numbers} as the numbers $n$ such that $T(T^\ast(n))=n^k$ for $k\geq 2$. It is our aim in this section to characterize these $k$-$\ktt$-perfect numbers.

Let $n=p_1^{\alpha_1}\cdots p_r^{\alpha_r}$ be the prime factorization of $n>1$. Then the number of unitary divisors of $n$, $\tau^\ast (n)=2^r$ and $T^\ast (n)=n^{2^{r-1}}$. Das and Saikia \cite{nntdm} mentioned that for $k$-$\ktt$-perfect number we must have 

\begin{equation}\label{E3-1}
2^r(\alpha_1\cdot 2^{r-1}+1)\cdots (\alpha_r\cdot 2^{r-1}+1)=4k
\end{equation}

\noindent for $k\geq 2$. In the following results we characterize these class of numbers. Let $n=p^{\alpha_1}_1\cdots p^{\alpha_r}_r$ with
$r\in\mathbb{N^\ast}$ be the prime factorization
of an integer $n>1$ where $p_i$ with $i\in\llbracket 1,r\rrbracket$
are prime numbers and $\alpha_i\in\mathbb{N^\ast}$ for all
$i\in\llbracket 1,r\rrbracket$.

\begin{theorem}\label{t3-1}
\hspace{2em}
\begin{enumerate}
\item All $2$-$T_0T^{\ast}$-perfect numbers have the form $n=p^3_1$;
\item All $3$-$T_0T^{\ast}$-perfect numbers have the form $n=p^5_1$;
\item All $4$-$T_0T^{\ast}$-perfect numbers have the form $n=p^7_1$;
\item All $5$-$T_0T^{\ast}$-perfect numbers have the form $n=p^9_1$;
\item All $6$-$T_0T^{\ast}$-perfect numbers have the form $n=p^{11}_1$;
\item All $7$-$T_0T^{\ast}$-perfect numbers have the form $n=p^{13}_1$;
\item All $8$-$T_0T^{\ast}$-perfect numbers have the form $n=p^{15}_1$;
\item All $9$-$T_0T^{\ast}$-perfect numbers have the form $n=p^{17}_1$
or $n=p_1p_2$;
\item All $10$-$T_0T^{\ast}$-perfect numbers have the form $n=p^{19}_1$;
\end{enumerate}
\end{theorem}

\begin{proof}

We will examine in detail the cases where $k=2,3,9$ leaving to the
reader the task to verify the other statements as the proofs are similar.

We first prove that all $2$-$T_0T^{\ast}$-perfect numbers have the
form $n=p^3_1$. In the following, we will investigate the different subcases
beginning from $r=1$.
\begin{enumerate}[label=$\bullet$]
\item $r=1$: (\ref{E3-1}) becomes $2\cdot(\alpha_1+1)=8$; it
  gives $\alpha_1=3$.
\item $r=2$: (\ref{E3-1}) becomes
  $4\cdot(2\alpha_1+1)\cdot(2\alpha_2+1)=8$ which is equivalent 
  to $(2\alpha_1+1)\cdot(2\alpha_2+1)=2$; it is not possible since
  $(2\alpha_1+1)\cdot(2\alpha_2+1)$ is odd whereas $2$ is even.
  \item $r=3$: (\ref{E3-1}) becomes
  $8\cdot(4\alpha_1+1)\cdot(4\alpha_2+1)\cdot(4\alpha_3+1)=8$ which is
  equivalent to $(4\alpha_1+1)\cdot(4\alpha_2+1)\cdot(4\alpha_3+1)=1$;
  it is not possible 
  since $\alpha_1,\alpha_2,\alpha_3\geq 1$ imply that
  $(4\alpha_1+1)\cdot(4\alpha_2+1)\cdot(4\alpha_3+1)\geq 125$.
\item $r\geq 4$: (\ref{E3-1}) becomes
  $2^{r-3}\cdot(\alpha_1\cdot 2^{r-1}+1)\cdots(\alpha_r\cdot
  2^{r-1}+1)=1$ which is not possible since 
  $2^{r-3}\!\not|1$ for $r\geq 4$.
\end{enumerate}
So, only the subcase where $r=1$ is valid, which means that if $k=2$,
then $n=p^3_1$.

Secondly, we now prove that all $3$-$T_0T^{\ast}$-perfect numbers have the
form $n=p^5_1$. In the following, we will investigate the different subcases
beginning from $r=1$.
\begin{enumerate}[label=$\bullet$]
\item $r=1$: (\ref{E3-1}) becomes $2\cdot(\alpha_1+1)=12$; it
  gives $\alpha_1=5$.
\item $r=2$: (\ref{E3-1}) becomes
  $4\cdot(2\alpha_1+1)\cdot(2\alpha_2+1)=12$ which is equivalent 
  to $(2\alpha_1+1)\cdot(2\alpha_2+1)=3$; it is not possible since
  for $\alpha_1,\alpha_2\geq 1$, we have
  $(2\alpha_1+1)\cdot(2\alpha_2+1)\geq 9$.
\item $r\geq 3$: (\ref{E3-1}) becomes
  $2^{r-2}\cdot(\alpha_1\cdot 2^{r-1}+1)\cdots(\alpha_r\cdot
  2^{r-1}+1)=3$ which is not possible since $2^{r-2}\!\not|3$ for $r\geq 3$.
\end{enumerate}
So, only the subcase where $r=1$ is valid, which means that if $k=3$,
then $n=p^5_1$.

Third, we prove that all $9$-$T_0T^{\ast}$-perfect numbers have the
form $n=p^{17}_1$ or $n=p_1p_2$. In the following, we will investigate
the different subcases beginning from $r=1$.
\begin{enumerate}[label=$\bullet$]
\item $r=1$: (\ref{E3-1}) becomes $2\cdot(\alpha_1+1)=36$; it
  gives $\alpha_1=17$.
\item $r=2$: (\ref{E3-1}) becomes
  $4\cdot(2\alpha_1+1)\cdot(2\alpha_2+1)=36$ which is equivalent 
  to $(2\alpha_1+1)\cdot(2\alpha_2+1)=9$; there is a trivial solution which
  is obtained when $\alpha_1=\alpha_2=1$, if at least one of the
  integers among the integers $\alpha_1$ and $\alpha_2$ is greater than $2$,
  then $(2\alpha_1+1)\cdot(2\alpha_2+1)>9$ implying that there is no other
  solution.
\item $r\geq 3$: (\ref{E3-1}) becomes
  $2^{r-2}\cdot(\alpha_1\cdot 2^{r-1}+1)\cdots(\alpha_r\cdot
  2^{r-1}+1)=9$ which is not possible since $2^{r-2}\!\not|9$ for
  $r\geq 3$.
\end{enumerate}
So, only the subcases where $r=1$ and $r=2$ with
$\alpha_1=\alpha_2=1$ is valid, which means that if $k=9$, then either $n=p^{17}_1$
or $n=p_1p_2$. 
\end{proof}

\begin{theorem}\label{t3-2}
All $p$-$T_0T^{\ast}$-perfect numbers for a prime $p$ have the form
$n=p^{2p-1}_1$.
\end{theorem}

\begin{proof}

According to (\ref{E3-1}), we must solve the equation
\be
\label{E3-2}
2^r\cdot(\alpha_1\cdot 2^{r-1}+1)\cdots(\alpha_r\cdot 2^{r-1}+1)=4p.
\ee
In the following, we will investigate the different cases beginning
from $r=1$. 
\begin{enumerate}[label=$\bullet$]
\item $r=1$: (\ref{E3-2}) becomes $2\cdot(\alpha_1+1)=4p$; it
  gives $\alpha_1=2p-1$.
\item $r=2$: (\ref{E3-2}) becomes
  $4\cdot(2\alpha_1+1)\cdot(2\alpha_2+1)=4p$; it gives 
  $(2\alpha_1+1)(2\alpha_2+1)=p$; notice that the conditions
  $\alpha_1,\alpha_2\geq 1$ imply that $(2\alpha_1+1)(2\alpha_2+1)\geq
  9$ and so in such a case, $p$ must be necessarily an odd prime
  number; notice also that if one of the numbers among the numbers
  $2\alpha_1+1,2\alpha_2+1$ is equal to $p$, it implies that the other
  number among the numbers $2\alpha_1+1,2\alpha_2+1$ is equal to $1$
  implying that one of the numbers among the numbers
  $\alpha_1,\alpha_2$ would be zero,
  which is not compatible with the conditions $\alpha_1,\alpha_2\geq
  1$; therefore, this case is not possible due to the fact that $p$ is
  a prime number which is squarefree. 
\item $r\geq 3$: (\ref{E3-2}) becomes
  $2^{r-2}\cdot(\alpha_1\cdot 2^{r-1}+1)\cdots(\alpha_r\cdot
  2^{r-1}+1)=p$ which is not possible since
  $2^{r-2}\!\not|p$ for $r\geq 3$ for odd prime $p$.
\end{enumerate}
So, only the case where $r=1$ is valid for which $n=p^{2p-1}_1$.
\end{proof}

\begin{theorem}\label{t3-22}
Let $k$ be a non-zero positive integer. Then we have the following:

\begin{enumerate}
    \item For any prime $p$, the integers whose prime decompositions have the form $p^{2p^k-1}_1$ are $p^k$-$T_0T^{\ast}$-perfect numbers.
    \item If $k\geq 2$, then for any odd prime $p$, the integers whose prime decompositions have the form $p^{(p^{a_1}-1)/2}_1\cdot p^{(p^{a_2}-1)/2}_2$ where $a_1,a_2\in\mathbb{N}^{\ast}$ such that $a_1+a_2=k$, are $p^k$-$T_0T^{\ast}$-perfect numbers.
    \item Any integer which doesn't have prime decomposition as $p^{2p^k-1}_1$ for prime $p$ or $p^{(p^{a_1}-1)/2}_1\cdot p^{(p^{a_2}-1)/2}_2$ for odd prime $p$, where $a_1,a_2\in\mathbb{N}^{\ast}$ such that $a_1+a_2=k$, are not $p^k$-$T_0T^{\ast}$-perfect numbers.
\end{enumerate}
  
\end{theorem}

\begin{proof}
According to Equation (\ref{E3-1}), we must solve
\begin{equation}\label{E3-22}
2^r\cdot(\alpha_1\cdot 2^{r-1}+1)\cdots(\alpha_r\cdot 2^{r-1}+1)=4p^k
\end{equation}
Next, we will examine the different cases beginning from $r=1$.
\begin{enumerate}[label=$\bullet$]
\item $r=1$: (\ref{E3-22}) becomes $2\cdot(\alpha_1+1)=4p^k$; it gives $\alpha_1=2p^k-1$, which proves the first statement.
\item $r=2$: (\ref{E3-22}) becomes $4\cdot(2\alpha_1+1)\cdot(2\alpha_2+1)=4p^k$; it gives $(2\alpha_1+1)\cdot(2\alpha_2+1)=p^k$; in this case, since $(2\alpha_1+1)\cdot(2\alpha_2+1)$ is odd whatever nonzero positive integers $\alpha_1,\alpha_2$ are, the prime $p$ must be odd; if $k=1$, then we have $(2\alpha_1+1)\cdot(2\alpha_2+1)=p$ which is not possible for $\alpha_1,\alpha_2\geq 1$ since $p$ is a prime number which is squarefee; if $k\geq 2$, then according to the fundamental theorem of arithmetic, since $2\alpha_i+1\geq 3$ with $\alpha_i\geq 1$ for $i=1,2$, the equation $(2\alpha_1+1)\cdot(2\alpha_2+1)=p^k$ implies that there exists two nonzero positive integers $a_1,a_2$ such that $2\alpha_1+1=p^{a_1}$ and $2\alpha_2+1=p^{a_2}$ and $a_1+a_2=k$, which proves the second statement.
\item $r\geq 3$: (\ref{E3-22}) becomes $2^{r-2}\cdot(\alpha_1\cdot 2^{r-1}+1)\cdots(\alpha_r\cdot 2^{r-1}+1)=p^k$; in this case, $p^k$ must be divisible by $2$ and since $p$ is prime, it entails that $p=2$; but then $\alpha_i\cdot 2^{r-1}+1$ for all $i\in\llbracket 1,r\rrbracket$ is divisible by $2$; that is impossible since $\alpha_i\cdot 2^{r-1}+1$ for all $i\in\llbracket 1,r\rrbracket$ is odd, which proves the third statement.
\end{enumerate}

\end{proof}

\begin{remark}\label{r3-22}
Let $k$ be an integer which is greater than $2$ and let $p$ be an odd prime number.
According to Theorem \ref{t3-22}, the integers of the form $p^{(p^{a}-1)/2}_1\cdot p^{(p^{k-a}-1)/2}_2$ with $a\in\llbracket 1,k-1\rrbracket$ are $p^k$-$T_0T^{\ast}$-perfect numbers.
\end{remark}

\begin{corollary}\label{c3-22}
\hspace{2em}
\begin{enumerate}
    \item For any prime $p$, the integers whose prime decompositions have the form $p^{2p^2-1}_1$, are $p^2$-$T_0T^{\ast}$-perfect numbers. 
    \item If $p$ is an odd prime number, then the integers whose prime decompositions have the form $p^{(p-1)/2}_1\cdot p^{(p-1)/2}_2$, are $p^2$-$T_0T^{\ast}$-perfect numbers.
    \item Any integer which doesn't have prime decomposition as $p^{2p^2-1}_1$ for a prime $p$ or $p^{(p-1)/2}_1\cdot p^{(p-1)/2}_2$ for an odd prime $p$, are not $p^2$-$T_0T^{\ast}$-perfect numbers.
\end{enumerate}
 
\end{corollary}

The first statement of Corollary \ref{c3-22}, follows directly from Theorem \ref{t3-22}. The second statement of Corollary \ref{c3-22}, follows from Remark \ref{r3-22} when $k=2$ and so $n=p^{(p^{a}-1)/2}_1\cdot p^{(p^{2-a}-1)/2}_2$ with $a$ a nonzero positive integers which verifies the condition $1\leq a<2$. In this case, there is only one possibility, namely $a=1$.

\begin{corollary}\label{c3-23}
Let $k$ be a nonzero positive integer. All $2^k$-$T_0T^{\ast}$-perfect numbers have the form $p^{2^{k+1}-1}_1$.
\end{corollary}

\noindent Corollary \ref{c3-23} follows directly from Theorem \ref{t3-22}.

\begin{theorem}\label{t3-23}
Let $k$ be an odd positive integer which is not prime. All $k$-$T_0T^{\ast}$-perfect numbers have the form $p^{2k-1}_1$ or $p^{(d-1)/2}_1\cdot p^{(k/d-1)/2}_2$ where $d$ is a positive proper divisor of $k$ which satisfies the inequalities $2<d<k$.
\end{theorem}

\begin{proof}
According to Equation (\ref{E3-1}), we must solve
\begin{equation}\label{E3-23}
2^r\cdot(\alpha_1\cdot 2^{r-1}+1)\cdots(\alpha_r\cdot 2^{r-1}+1)=4k.
\end{equation}
Next, we will examine the different cases beginning from $r=1$.
\begin{enumerate}[label=$\bullet$]
\item $r=1$: (\ref{E3-23}) becomes $2\cdot(\alpha_1+1)=4k$; it gives $\alpha_1=2k-1$.
\item $r=2$: (\ref{E3-23}) becomes $4\cdot(2\alpha_1+1)\cdot(2\alpha_2+1)=4k$; it gives $(2\alpha_1+1)\cdot(2\alpha_2+1)=k$; in this case, let $k=q^{\beta_1}_1\cdots q^{\beta_s}_s$ be the prime decomposition of the odd integer $k$ where $q_i$ is an odd prime for all $i\in\llbracket 1,s\rrbracket$ with $s\in\mathbb{N}^{\ast}$ and $\beta_i\in\mathbb{N}^{\ast}$ for all $i\in\llbracket 1,s\rrbracket$ with $s\in\mathbb{N}^{\ast}$; according to the fundamental theorem of arithmetic, we have $2\alpha_1+1=q^{\gamma_1}_1\cdots q^{\gamma_s}_s$ and $2\alpha_2+1=q^{\delta_1}_1\cdots q^{\delta_s}_s$ where for all $i\in\llbracket 1,s\rrbracket$ with $s\in\mathbb{N}^{\ast}$, $\gamma_i,\delta_i\in\mathbb{N}$ such that $\gamma_i+\delta_i=\beta_i$; since $\alpha_1,\alpha_2\geq 1$ ($r=2$), there exists at least one integer among the integers $\gamma_i$ ($i=1,\ldots,s$) which is greater than $1$ and there exists at least one integer among the integers $\delta_i$ ($i=1,\ldots,s$) which is greater than $1$; we must have $2<q^{\gamma_1}_1\cdots q^{\gamma_s}_s<k$ and $2<q^{\delta_1}_1\cdots q^{\delta_s}_s<k$; notice that we have $q^{\delta_1}_1\cdots q^{\delta_s}_s=k/(q^{\gamma_1}_1\cdots q^{\gamma_s}_s)$ which follows from the equation $(2\alpha_1+1)\cdot(2\alpha_2+1)=k$; in the following, we set $d=q^{\gamma_1}_1\cdots q^{\gamma_s}_s$ and so $q^{\delta_1}_1\cdots q^{\delta_s}_s=k/d$; accordingly, $d$ is a positive proper divisor of $k$, $2\alpha_1+1=d$ and $2\alpha_2+1=k/d$; it gives $\alpha_1=(d-1)/2$ and $\alpha_2=(k/d-1)/2$.
\item $r\geq 3$: (\ref{E3-23}) becomes $2^{r-2}\cdot(\alpha_1\cdot 2^{r-1}+1)\cdots(\alpha_r\cdot 2^{r-1}+1)=k$; in this case, $k$ must be divisible by $2$; which is impossible since $k$ is odd. 
\end{enumerate}
So, only the cases where $r=1$ and $r=2$ are possible. This completes the proof.
\end{proof}

\begin{remark}
We can notice that in Theorem \ref{t3-23}, if $d$ is a positive proper divisor of $k$, then $k/d$ is also a positive proper divisor of $k$. For instance, if $k=p^2$ where $p$ is an odd prime, then there is only one possibility for $d$ which is consistent with Theorem \ref{t3-23}, namely $d=k/d=p$. Notice that in this case, using Theorem \ref{t3-23}, the two results stated in Corollary \ref{c3-22} can be recovered. Another case which illustrates the fact mentioned at the beginning of this remark, is when $k=pq$ wher $p$ and $q$ are odd prime numbers. Then the positive proper divisors of $k$ are $p$ and $q$. In this case, Theorem \ref{t3-23} implies that all $pq$-$T_0T^{\ast}$-perfect number for odd prime numbers $p$ and $q$, have the form $p^{2pq-1}_1$ or $p^{(p-1)/2}_1\cdot p^{(q-1)/2}_2$.
\end{remark}

\begin{definition}
Let $n$ be an integer which is greater than $1$. A \textbf{multiplicative partition} (\seqnum{A001055}) or \textbf{unordered factorization} of $n$ is a decomposition of $n$ into a product of integers which belong to $\llbracket 1,n\rrbracket$, where the order of terms is irrelevant.
\end{definition}

\begin{theorem}\label{t3-24}
Let $k$ be an odd positive integer and let $m$ be a nonzero positive integer. Then we have the following:
\begin{enumerate}
    \item  The integers whose prime decompositions have the form $p^{2^{m+1}k-1}_1$, are $2^m k$-$T_0T^{\ast}$-perfect numbers.
    \item For $k>1$, if there exist integers $d_1,\ldots,d_{m+2}$ which form a multiplicative partition of $k$ such that for all $i\in\llbracket 1,m+2\rrbracket$, $d_i\equiv 1\pmod{2^{m+1}}$, then the integers whose prime decompositions have the form $p^{(d_1-1)/2^{m+1}}_1\cdots p^{(d_{m+2}-1)/2^{m+1}}_{m+2}$, are $2^m k$-$T_0T^{\ast}$-perfect numbers. 
    \item Any integer which doesn't have prime decomposition as $p^{2^{m+1}k-1}_1$ or \\ $p^{(d_1-1)/2^{m+1}}_1\cdots p^{(d_{m+2}-1)/2^{m+1}}_{m+2}$, are not $2^m k$-$T_0T^{\ast}$-perfect numbers.
\end{enumerate}
 
\end{theorem}

\begin{proof}
According to Equation (\ref{E3-1}), we must solve
\begin{equation}\label{E3-24}
2^r\cdot(\alpha_1\cdot 2^{r-1}+1)\cdots(\alpha_r\cdot 2^{r-1}+1)=2^{m+2}k
\end{equation}
Next, we will examine the different cases beginning from $r=1$.
\begin{enumerate}[label=$\bullet$]
\item $r=1$: (\ref{E3-24}) becomes $2\cdot(\alpha_1+1)=2^{m+2}k$; it gives $\alpha_1=2^{m+1}k-1$.
\item $2\leq r<m+2$: (\ref{E3-24}) becomes $2^r\cdot(\alpha_1\cdot 2^{r-1}+1)\cdots(\alpha_r\cdot 2^{r-1}+1)=2^{m+2}k$; it gives $(\alpha_1\cdot 2^{r-1}+1)\cdots(\alpha_r\cdot 2^{r-1}+1)=2^{m+2-r}k$; 
which is impossible since $(\alpha_1\cdot 2^{r-1}+1)\cdots(\alpha_r\cdot 2^{r-1}+1)$ is odd whereas $2^{m+2-r}k$ for $2\leq r<m+2$ is even.
\item $r=m+2$: (\ref{E3-24}) becomes $(\alpha_1\cdot 2^{m+1}+1)\cdots(\alpha_{m+2}\cdot 2^{m+1}+1)=k$; if $k=1$, then it is not possible since $(\alpha_1\cdot 2^{m+1}+1)\cdots(\alpha_{m+2}\cdot 2^{m+1}+1)>2^{(m+2)(m+1)}>1$ for $m\in\mathbb{N}^{\ast}$; so, $k>1$; in this case, let $k=q^{\beta_1}_1\cdots q^{\beta_s}_s$ be the prime decomposition of the odd integer $k$ where $q_i$ is an odd prime for all $i\in\llbracket 1,s\rrbracket$ with $s\in\mathbb{N}^{\ast}$ and $\beta_i\in\mathbb{N}^{\ast}$ for all $i\in\llbracket 1,s\rrbracket$ with $s\in\mathbb{N}^{\ast}$; then according to the fundamental theorem of arithmetic, for all $i\in\llbracket 1,m+2\rrbracket$, we have $2\alpha_i+1=q^{\gamma_{i,1}}_1\cdots q^{\gamma_{i,s}}_s$ and $\beta_i=\sum^s_{j=1}\gamma_{i,j}$; since for all $i\in\llbracket 1,m+2\rrbracket$, $\alpha_i\geq 1$, for given $i\in\llbracket 1,m+2\rrbracket$, there exists at least one integer among the integers $\gamma_{i,j}$ ($j=1,\ldots,s$) which is greater than $1$; we must have $2<q^{\gamma_{i,1}}_1\cdots q^{\gamma_{i,s}}_s<k$; in the following, for all $i\in\llbracket 1,m+2\rrbracket$, we set $d_i=q^{\gamma_{i,1}}_1\cdots q^{\gamma_{i,s}}_s$; accordingly, for all $i\in\llbracket 1,m+2\rrbracket$, $d_i$ is a positive proper divisor of $k$; we have also $k=d_1\cdots d_{m+2}$ and for all $i\in\llbracket 1,m+2\rrbracket$, $\alpha_i\cdot 2^{m+1}+1=d_i$; provided for all $i\in\llbracket 1,m+2\rrbracket$, $d_i-1$ is divisible by $2^{m+1}$, it gives $\alpha_i=(d_i-1)/2^{m+1}$.
\item $r>m+2$: (\ref{E3-24}) becomes $2^{r-(m+2)}\cdot(\alpha_1\cdot 2^{r-1}+1)\cdots(\alpha_r\cdot 2^{r-1}+1)=k$; which is impossible since $2^{r-(m+2)}\cdot(\alpha_1\cdot 2^{r-1}+1)\cdots(\alpha_r\cdot 2^{r-1}+1)$ for $r>m+2$ is even whereas $k$ is assumed to be odd. 
\end{enumerate}
So, only the case where $r=1$ for odd positive integer $k$ and for nonzero positive integer $m$ and the case where $r=m+2$ for odd positive integer $k$ which has a multiplicative partition whose members are congruent to $1$ modulo $2^{m+1}$ for $m\in\mathbb{N}^{\ast}$, are possible. This completes the proof.
\end{proof}

\begin{corollary}\label{c3-24}
Let $m$ be a nonzero positive integer. All $2^m(2^{m+1}+1)^{m+2}$-$T_0T^{\ast}$-perfect numbers have the form $p^{2^{m+1}(2^{m+1}+1)^{m+2}-1}_1$ or $p_1\cdots p_{m+2}$.
\end{corollary}

\noindent Corollary \ref{c3-24} follows directly from Theorem \ref{t3-24} by taking $k=(2^{m+1}+1)^{m+2}$ for $m\in\mathbb{N}^{\ast}$. Notice that $k$ is odd here.

\begin{theorem}\label{t3-3}
Let $p$ be a prime, with $2^p-1$ being a Mersenne prime. Then
$2^{p-1}\cdot(2^p-1)$ is the only even perfect number which is a
$3\cdot(2p-1)$-$T_0T^{\ast}$-perfect number.
\end{theorem}

\begin{proof}

According to (\ref{E3-1}), we must solve the equation
\be
\label{E3-3}
2^r\cdot(\alpha_1\cdot 2^{r-1}+1)\cdots(\alpha_r\cdot 2^{r-1}+1)=12\cdot(2p-1)
\ee
where $p$ is a prime such that $2^p-1$ is a Mersenne prime.

First we find the possible forms of all
$3\cdot(2p-1)$-$T_0T^{\ast}$-perfect numbers.
In the following, we will investigate the different cases beginning
from $r=1$. 
\begin{enumerate}[label=$\bullet$]
\item $r=1$: (\ref{E3-3}) becomes
  $2\cdot(\alpha_1+1)=12\cdot(2p-1)$; it gives $\alpha_1=12p-5$.
\item $r=2$: (\ref{E3-3}) becomes
  $4\cdot(2\alpha_1+1)\cdot(2\alpha_2+1)=12\cdot(2p-1)$ which is equivalent 
  to $(2\alpha_1+1)\cdot(2\alpha_2+1)=3\cdot(2p-1)$; a trivial solution is
  given by $\alpha_1=1$ and $\alpha_2=p-1$ (notice that a particular subcase of
  this trivial solution is $\alpha_1=\alpha_2=1$ which is consistent
  with (\ref{E3-3}) only if $p=2$);
  more generally, we must have either $2\alpha_1+1\equiv 0$ (mod $3$) or
  $2\alpha_2+1\equiv 0$ (mod $3$); without loss of generality, let us take
  $2\alpha_2+1=3k$ with $k$ an odd positive integer which divides
  $2p-1$ so that (\ref{E3-3}) is satisfied; then it is
  not difficult to see that $2\alpha_1+1=\frac{2p-1}{k}$ and we obtain
  for $\alpha_1$ and $\alpha_2$, the parametrisation
  $\alpha_1(k)=\frac{2p-(k+1)}{2k}$ and $\alpha_2(k)=\frac{3k-1}{2}$; notice
  that the subcase $\alpha_1=p-1$ and $\alpha_2=1$ is recovered when $k=1$.  
\item $r\geq 3$: (\ref{E3-3}) becomes
  $2^{r-2}\cdot(\alpha_1\cdot 2^{r-1}+1)\cdots(\alpha_r\cdot
  2^{r-1}+1)=3\cdot(2p-1)$ which is not possible since
  $2^{r-2}\!\not|3\cdot(2p-1)$ for $r\geq 3$.
\end{enumerate}
So, only the case where $r=1$ for which $n=p^{12p-5}_1$
and the case where $r=2$ for which $n=p^{\frac{2p-(k+1)}{2k}}_1\cdot
p^{\frac{3k-1}{2}}_2$ such that $k$ is an odd positive integer which
divides $2p-1$ are valid.

In particular, when $r=2$, if $k=1$, then we get
$n=p^{p-1}_1\cdot p_2$. Conversely, when $r=2$, if $\alpha_1(k)=p-1$ and $\alpha_2(k)=1$, then using the parametrisation given above for $\alpha_1$ and $\alpha_2$
when $r=2$, we have $\dfrac{2p=(k+1)}{2k}=p-1$ and $\dfrac{3k-1}{2}=1$. This implies that $k=1$. Therefore, when $r=2$
$$
n=p^{p-1}_1\cdot p_2\,\Leftrightarrow\,\left\{\begin{array}{c}
\alpha_1(k)=p-1;
\\
\alpha_2(k)=1.
\end{array}\right.\,\Leftrightarrow\,k=1.
$$

Secondly we see if there exists an even perfect number which is
$3\cdot(2p-1)$-$T_0T^{\ast}$-perfect number. According to
the Euclid-Euler theorem, an even perfect number takes the form
$2^{q-1}\cdot(2^q-1)$, where $q$ is a prime such that $2^q-1$ is a
Mersenne prime. Accordingly, an even perfect number corresponds to
the case where $r=2$ which was investigated above. So,
$n=p^{\alpha_1}_1\cdot p^{\alpha_2}_2$ is an even perfect number if and only if
$$
p^{\alpha_1}_1\cdot p^{\alpha_2}_2=2^{q-1}\cdot(2^q-1).
$$

Since the prime factorization of an integer is unique up to the order
of the prime factors, without loss of generality, taking
$1\leq\alpha_2\leq\alpha_1$, since $2$ and
$2^q-1$ are prime, we will have $p_1=2$, $p_2=2^q-1$ with $\alpha_1(k)=q-1$ and $\alpha_2(k)=1$. Using the parametrization introduced above when $r=2$, this system is
equivalent to $\dfrac{2p-(k+1)}{2k}=q-1$ and $\dfrac{3k-1}{2}=1$. Solving this system for $(k,p)$, it results that $k=1$ and $p=q$.
Since when $r=2$
$$
n=p^{p-1}_1\cdot p_2\,\Leftrightarrow\,\left\{\begin{array}{c}
\alpha_1(k)=p-1;
\\
\alpha_2(k)=1.
\end{array}\right.\,\Leftrightarrow\,k=1.
$$

Let us finish the proof by studying the subcase where both $2\alpha_1+1$ and $2\alpha_2+1$ are divisible by $3$.
If both $2\alpha_1+1$ and
$2\alpha_2+1$ are divisible by $3$, then there 
exist two odd positive integers $k_1$ and $k_2$ such that
$2\alpha_1+1=3k_1$ and $2\alpha_2+1=3k_2$. It gives $\alpha_1=\dfrac{3k_1-1}{2}$ and $\alpha_2=\dfrac{3k_2-1}{2}$. Accordingly, the equation $(2\alpha_1+1)\cdot(2\alpha_2+1)=3\cdot(2p-1)$ becomes $9k_1k_2=3\cdot(2p-1)$. After simplification, we get $3k_1k_2=2p-1$.

If $n$ is an even perfect number, then $n=2^{p-1}\cdot(2^p-1)$ and
without loss of generality, (since for $r=2$, $n=p^{\alpha_1}_1\cdot
p^{\alpha_2}_2$ and since $2^p-1$ is a Mersenne prime)
we can set $p_1=2$ and $p_2=2^p-1$. Then we have
$$
\alpha_1=\frac{3k_1-1}{2}=p-1\,\Rightarrow\,3k_1=2p-1
$$
and
$$
\alpha_2=\frac{3k_2-1}{2}=1\,\Rightarrow\,k_2=1.
$$
It is consistent with the equation $3k_1k_2=2p-1$ and the fact that $k_1$, $k_2$ are odd positive integers. It means that
$n=p^{p-1}_1p_2$ as expected. In fact, using the parametrisation
$$
\alpha_1(k)=\frac{2p-(k+1)}{2k}
$$
$$
\alpha_2(k)=\frac{3k-1}{2}
$$
and taking by identification $k=k_2$ and $k_1=\frac{2p-1}{3k}$, because
$3k_1k_2=2p-1$ when both $2\alpha_1+1$ and $2\alpha_2+1$ are
divisible by $3$, this subcase can be recovered. In particular, if
$k=1$, then it is consistent with the statement of Theorem \ref{t3-3}. Notice that if both  $2\alpha_1+1$ and $2\alpha_2+1$ are divisible by $3$, then $p\equiv 2\pmod 3$. Indeed, since $k_1$ is an odd positive integer, there exists an integer $m_1$ such that $k_1=2m_1+1$. Or,
$3k_1=2p-1$. So, $2p=3k_1+1=6m_1+4$. It gives $p=3m_1+2$. Notice also
that $m_1$ shall be an odd positive integer since $p$ is necessarily
prime so that $2^p-1$ be a Mersenne prime.

We conclude that $2^{p-1}\cdot(2^p-1)$ is the only even perfect number
which is a $3\cdot(2p-1)$-$T_0T^{\ast}$-perfect number.
\end{proof}

In this section, we have presented results only on the canonical representation of $k$-$\ktt$ perfect numbers. Other techniques may be used to derive results on the bounds of such numbers, but here we do not proceed in that direction.

\section{{Acknowledgements}}

The authors are grateful to an anonymous referee for valuable comments and for pointing them to some references. The second author is grateful to the excellent facilities provided to him by The Abdus Salam International Centre for Theoretical Physics, Trieste (Italy), where this work was initiated.

\bibliographystyle{amsplain}

\begin{thebibliography}{10}



\bibitem{bege} A. Bege, On multiplicatively unitary perfect numbers, {\it Seminar on Fixed Point Theory, Cluj-Napoca}, \textbf{2} (2001), 59--63.

\bibitem {nntdm} B. Das and H. K. Saikia, On generalized multiplicative perfect numbers, {\it Notes. Numb. Thy. Disc. Math.}, \textbf{19} (2013), 37--42.

\bibitem{kanold} H. J. Kanold, \"Uber super-perfect numbers, {\it Elem. Math.}, \textbf{24} (1969), 61--62.

\bibitem{nc} J.-L. Nicolas and G. Robin,  Majorations explicites pour le nombre de diviseurs de $n$, {\it Canad. Math. Bull.}, \textbf{26} (1983), 485--492.

\bibitem{sandor1} J. S\'andor, On multiplicatively perfect numbers, {\it J. Ineq. Pure Appl. Math.}, \textbf{2} (2001), Art. 3, 6pp.

\bibitem{sandor} J. S\'andor, On multiplicatively $e$-perfect numbers, {\it J. Ineq. Pure Appl. Math.}, \textbf{5} (2004), 4pp.

\bibitem{sandor2} J. S\'andor, A note on exponential divisors and related arithmetic functions, {\it Scientia Magna}, \textbf{1} (1) (2005), 97--101.

\bibitem{sandor3} J. S\'andor, On exponentially harmonic numbers, {\it Scientia Magna}, \textbf{2} (3) (2006), 44--47.

\bibitem{sandornt} J. S\'andor and B. Crstici, Handbook of Number Theory II, Kluwer Academic Publishers, 2005.

\bibitem{sandor22} J. S\'andor, D. S. Mitrinovi\'c and B. Crstici, Handbook of Number Theory I (2nd ed.), Springer, 2006.

\bibitem{subba} E. G. Straus and M. V. Subbarao, On exponential divisors, {\it Duke Math. J.}, \textbf{41} (1974), 465--471.

\bibitem{surya} D. Suryanarayana, Super-perfect numbers, {\it Elem. Math.}, \textbf{24} (1969), 16--17.




\end{thebibliography}

\end{document}